\documentclass[11pt]{amsart}

\usepackage[T1]{fontenc}
\usepackage{lmodern}
\usepackage{microtype}
\usepackage{amsmath,amssymb,amsthm,mathtools}
\usepackage{aliascnt}
\usepackage{enumitem}
\usepackage{hyperref}
\usepackage[nameinlink,capitalize]{cleveref}
\usepackage{mathrsfs}

\hypersetup{
  colorlinks=true,
  linkcolor=blue,
  citecolor=blue,
  urlcolor=blue
}

\numberwithin{equation}{section}

\newtheorem{theorem}{Theorem}[section]

\newaliascnt{proposition}{theorem}
\newtheorem{proposition}[proposition]{Proposition}
\aliascntresetthe{proposition}

\newaliascnt{lemma}{theorem}
\newtheorem{lemma}[lemma]{Lemma}
\aliascntresetthe{lemma}

\newaliascnt{corollary}{theorem}
\newtheorem{corollary}[corollary]{Corollary}
\aliascntresetthe{corollary}

\newaliascnt{criterion}{theorem}

\aliascntresetthe{criterion}

\theoremstyle{definition}
\newaliascnt{definition}{theorem}

\aliascntresetthe{definition}

\newaliascnt{example}{theorem}
\newtheorem{example}[example]{Example}
\aliascntresetthe{example}

\theoremstyle{remark}
\newaliascnt{remark}{theorem}
\newtheorem{remark}[remark]{Remark}
\aliascntresetthe{remark}

\crefname{theorem}{Theorem}{Theorems}
\crefname{proposition}{Proposition}{Propositions}
\crefname{lemma}{Lemma}{Lemmas}
\crefname{corollary}{Corollary}{Corollaries}
\crefname{criterion}{Criterion}{Criteria}
\crefname{definition}{Definition}{Definitions}
\crefname{example}{Example}{Examples}
\crefname{remark}{Remark}{Remarks}

\DeclareMathOperator{\Sym}{Sym}
\DeclareMathOperator{\Sing}{Sing}

\DeclareMathOperator{\Spec}{Spec}
\DeclareMathOperator{\Tor}{Tor}

\newcommand{\PP}{\mathbb P}
\newcommand{\CC}{\mathbb C}
\newcommand{\OO}{\mathcal O}

\newcommand{\cS}{\mathcal S}
\newcommand{\cP}{\mathcal P}
\newcommand{\Om}{\Omega}
\newcommand{\reflex}[1]{\left(#1\right)^{\vee\vee}}

\title[Symmetric differentials on weighted hypersurfaces]
{Symmetric differentials on weighted hypersurfaces}

\author{Wan-Yuan Xu}
\address{Department of Mathematics, Shanghai Normal University, Shanghai 200234, P.\ R.\ China}
\email{wanyuanxu@shnu.edu.cn}


\begin{document}

\begin{abstract}
We study symmetric differentials on weighted hypersurfaces.  Let $k$ be an algebraically closed field and let
\[
X_d\subset \PP_k(a_0,\ldots,a_N),\qquad N\ge3,
\]
be a weighted hypersurface, with $a_{\min}:=\min_i a_i$.  In characteristic zero we prove twisted vanishing in the range $t<m a_{\min}$ for quasi-smooth hypersurfaces of degree $d\ge2a_{\min}$, with the corresponding reflexive statement on the coarse space.  We also prove a tame positive-characteristic analogue; in particular, reflexive symmetric differentials of every positive order vanish.  In the second part we specialize to $k=\CC$ and to surfaces in weighted projective three space.  Using the quotient-singularity criterion of Asega--De Oliveira--Weiss together with weighted Kummer and Segre constructions, we produce singular weighted hypersurfaces whose minimal resolutions have big cotangent bundle.  Combining the vanishing and bigness results with simultaneous resolution yields deformation-equivalent smooth projective surfaces for which the symmetric plurigenera jump from zero in every positive order to cubic asymptotic growth.
\end{abstract}

%

\maketitle

\section{Introduction}

Let $k$ be an algebraically closed field and let $X$ be a smooth projective variety over $k$.  A symmetric differential of order $m$ on $X$ is a section of $\Sym^m\Om_X^1$.  Over $k=\CC$, symmetric differentials occur naturally in the study of positivity of the cotangent bundle and hyperbolicity \cite{Bogomolov,BdO2006,Br17,CDR,DR,RR2013,RR2014}, and they play a basic role in the Green--Griffiths--Lang philosophy \cite{Bogomolov,McQuillan}.  Their existence is also sensitive to birational and deformation-theoretic phenomena \cite{BdO2006,BdO2008}.

Over $k=\CC$, for smooth subvarieties of ordinary projective space $\PP_{\CC}^N$ there are strong
vanishing results.  Br\"uckmann proved that a smooth hypersurface in $\PP_{\CC}^N$ carries no nonzero global symmetric differentials of positive order
\cite{Bruckmann}.  Sakai obtained related vanishing theorems for complete
intersections \cite{Sakai}, and Schneider proved the more general twisted
vanishing \cite{Schneider}
\[
 H^0\!\left(X,\Sym^m\Om_X^1\otimes\OO_X(t)\right)=0,
 \qquad t<m,
\]
for submanifolds $X\subset\PP_{\CC}^N$ with $\dim X>\frac{N}{2}$.
Bogomolov and De Oliveira later gave a geometric approach to twisted
symmetric differentials and used singular hypersurfaces to exhibit striking
non-vanishing and deformation-jumping phenomena \cite{BdO2008}.

The positive-characteristic theory has a parallel history.  Over a field $k$ of characteristic $p>0$, Br\"uckmann--Rackwitz \cite{BR90} proved that for a smooth hypersurface of dimension at least two one obtains
\[
H^0(X,\Sym^m\Om_X^1)=0 \qquad (p>m).
\]
Related work in prime characteristic includes the birational invariants of Br\"uckmann--M\"uller \cite{BM98}, the characteristic-dependent cohomology calculations of Br\"uckmann--Rackwitz \cite{BR00}, and Br\"uckmann's extension to Frobenius pullbacks of tensor differential forms \cite{Brueckmann2002}.  More recently, Gounelas--Liedtke proved that every smooth hypersurface of dimension at least two in ordinary projective space has no nonzero symmetric differential of positive order in arbitrary characteristic \cite[Proposition~8.1]{GounelasLiedtke}.

The purpose of this paper is to study both sides of this picture for weighted hypersurfaces.  The vanishing theorem is proved in arbitrary dimension: throughout the first part we let
\[
 P=\PP_k(a_0,\ldots,a_N),\qquad N\ge3,
\]
and consider hypersurfaces of dimension \(N-1\ge2\).  We refer to
\cite{Dolgachev,Fletcher} for the basic theory of weighted projective spaces
and weighted hypersurfaces.

Our first result treats smooth and quasi-smooth weighted hypersurfaces over an algebraically closed field of characteristic zero in their natural forms.  For smooth coarse hypersurfaces we obtain the usual sheaf-theoretic vanishing, whereas for quasi-smooth hypersurfaces we obtain a vanishing theorem for reflexive symmetric differentials. 
\begin{theorem}[Characteristic-zero vanishing]\label{thm:intro-vanishing}
Let $k$ be an algebraically closed field of characteristic zero and let
\[
P=\PP_k(a_0,\ldots,a_N),\qquad N\ge3,
\]
be well formed.  Let $S\subset P$ be a hypersurface of weighted degree $d$, and set
\[
W:=\sum_{i=0}^N a_i,
\qquad
a_{\min}:=\min_{0\le i\le N}a_i.
\]
Assume $d\ge2a_{\min}$.  Then the following hold.
\begin{enumerate}[label=\textup{(\roman*)}]
\item If $S$ is smooth and contained in $P_{\rm reg}$, then for every
$m\ge1$ and every integer $t<m a_{\min}$,
\[
 H^0\!\left(S,\Sym^m\Om_S^1\otimes\OO_S(t)\right)=0.
\]
In particular $H^0(S,\Sym^m\Om_S^1)=0$ for every $m\ge1$.
\item If $S$ is quasi-smooth and well formed, then for every $m\ge1$ and every integer
$t<m a_{\min}$,
\[
 H^0\!\left(S,
 \left(\Sym^m\Om_S^1\otimes\OO_S(t)\right)^{\vee\vee}\right)=0.
\]
In particular
\[
 H^0(S,\Sym^{[m]}\Om_S^1)=0,
 \qquad
 \Sym^{[m]}\Om_S^1:=\left(\Sym^m\Om_S^1\right)^{\vee\vee}.
\]
\end{enumerate}
When $a_{\min}=1$, the twisted range is precisely $t<m$.
\end{theorem}

Our second main theorem gives a tame weighted analogue for well-formed quasi-smooth weighted hypersurfaces, together with an explicit twisted vanishing range in positive characteristic.  We impose the convenient tameness hypothesis $p\nmid a_i$ for every weight to avoid the wild strata. 

\begin{theorem}[Positive-characteristic vanishing]\label{thm:intro-poschar}
Let $k$ be an algebraically closed field of characteristic $p>0$, let
\[
P=\PP_k(a_0,\ldots,a_N),\qquad N\ge3,
\]
be well formed, and assume $p\nmid a_i$ for every $i$.  Let $S\subset P$ be a well-formed quasi-smooth hypersurface of weighted degree $d$, and set
\[
a_{\min}:=\min_i a_i.
\]
Assume $d\ge2a_{\min}$.  For $m\ge1$, let $\rho_m$ be the remainder of $m-1$ modulo $p$, so that $0\le\rho_m<p$.  Then
\[
H^0\!\left(
S,
\left(\Sym^m\Om_S^1\otimes\OO_S(t)\right)^{\vee\vee}
\right)=0
\]
for every integer
\[
t<\min\{m a_{\min},\ d+\rho_m a_{\min}\}.
\]
In particular,
\[
H^0(S,\Sym^{[m]}\Om_S^1)=0
\qquad (m\ge1).
\]
If $m\le p$, the twisted range is the same as in characteristic zero, namely $t<m a_{\min}$; in particular, when $a_{\min}=1$ it is $t<m$.
\end{theorem}

The proof and the comparison with the preceding positive-characteristic results are given in \cref{sec:vanishing}.

The weighted setting is not a formal reformulation of the ordinary projective one.  A quasi-smooth weighted hypersurface is naturally smooth as a Deligne--Mumford stack, while its coarse space may have quotient singularities along higher-codimensional strata.  Consequently the
naive sheaf $\Om_S^1$ and the reflexive differential sheaf
$\Om_S^{[1]}:=\reflex{\Om_S^1}$ need not have the same global sections:
$\Om_S^1$ may contain torsion supported at the quotient singularities; compare
\cite{GKKP} for reflexive differentials on singular varieties.  Even in the smooth case, the vanishing theorem is not obtained simply by
embedding \(S\) into an ordinary projective space and applying Schneider's
theorem.  Indeed, such an embedding is induced only by a sufficiently
divisible multiple \(\OO_S(q)\) of the natural weighted polarization
\(\OO_S(1)\), and typically lands in a projective space whose dimension is
too large for Schneider's dimension hypothesis.  Moreover, even when the
dimension condition is satisfied, the resulting vanishing is formulated
with respect to the embedding polarization \(\OO_S(q)\) and therefore does
not recover the full range
\[
H^0\!\left(S,\Sym^m\Om_S^1\otimes\OO_S(t)\right)=0,
\qquad t<m a_{\min},
\]
with respect to the primitive weighted polarization \(\OO_S(1)\).
The proof given here instead works directly with the weighted Euler and
conormal sequences and the cohomology of weighted line bundles.

The proof is carried out on the smooth hypersurface stack; in the
quasi-smooth case, the resulting vanishing is then transferred to the coarse
space by the stack-to-coarse comparison.

\begin{remark}The reflexive formulation in \cref{thm:intro-vanishing}(ii) cannot in general
be replaced by the naive K\"ahler differential sheaf.  We exhibit a
quasi-smooth surface
\[
 S_{16}\subset \PP(1,1,2,5)
\]
with a unique cyclic quotient singularity of type $\frac15(1,2)$ for which
\[
 H^0(S_{16},\Om_{S_{16}}^1)\neq0,
 \qquad
 H^0(S_{16},\Om_{S_{16}}^{[1]})=0.
\]
The nonzero K\"ahler differential is torsion and is supported at the singular
point, hence disappears after taking the reflexive hull.  Thus reflexive
differentials are not merely a convenient replacement on the coarse space:
they are precisely the differentials detected by the smooth stack.
\end{remark}

The second part of the paper concerns the opposite phenomenon.  We specialize to $k=\CC,\,N=3$, so that all hypersurfaces under consideration are surfaces in weighted projective three spaces.  Let $Y$ be a
normal projective surface with quotient singularities and let
$\widetilde Y\to Y$ be its minimal resolution.  Asega, De Oliveira and Weiss
associate a local asymptotic invariant $h^1_\Omega(y)$ to every quotient
surface singularity and prove the following sufficient condition
\cite{ADW1,ADW2}:
\[
 \sum_{y\in\Sing(Y)}h^1_\Omega(y)
 >
 \frac{c_2(\widetilde Y)-c_1^2(\widetilde Y)}6
 \quad\Longrightarrow\quad
 \Om_{\widetilde Y}^1\ \text{is big}.
\]

We combine this criterion with Kummer maps from weighted projective spaces to
$\PP^3$.  The most concrete consequence is the following.

\begin{theorem}\label{thm:intro-A5}
There exists a degree-$12$ hypersurface
\[
 Y_{12}\subset\PP(1,1,1,2)
\]
with exactly $60$ singularities of type $A_5$ and no other singularities
such that the minimal resolution $\widetilde Y_{12}$ has big cotangent bundle.
\end{theorem}

The construction starts from the degree $6$ cyclic cover of $\PP^2$
branched along $6$ lines in general position, a sextic surface in $\PP^3$
with $15$ singularities of type $A_5$ \cite[Proof of Theorem~1.5]{ADW1}.
After a general change of coordinates, the degree-$4$ weighted Kummer
pullback to $\PP(1,1,1,2)$ has weighted degree $12$ and replaces those
$15$ singularities by $60$ singularities of the same analytic type.  The
general numerical mechanism behind this example is developed in
\cref{sec:Kummer}.

Finally, a weighted version of Segre's classical construction
\cite{Segre} gives a uniform existence statement.

\begin{theorem}\label{thm:intro-Segre}
For every integer $r\ge2$ there exist nodal hypersurfaces
\[
 X\subset\PP(1,1,1,r)
\]
whose minimal resolutions have big cotangent bundle.
\end{theorem}

The quotient-singularity bigness criterion is the numerical input in
\cref{thm:intro-A5,thm:intro-Segre}, but the weighted geometry is what makes
it effective here.  A weighted Kummer map multiplies prescribed quotient
singularities while keeping the degree and Chern numbers explicit, whereas
the weighted Segre construction produces, uniformly in $r$, enough nodes to
overcome the global Chern-number threshold.  These constructions therefore
place special singular members with big cotangent resolutions inside weighted
linear systems whose smooth members satisfy the vanishing theorem
\cref{thm:intro-vanishing}.  Simultaneous resolution of a Du Val smoothing
\cite{Brieskorn} then allows these two results to be compared inside a smooth
family.  This leads to the following weighted-projective deformation-jumping
statement, in the spirit of Bogomolov--De Oliveira
\cite{BdO2006,BdO2008}.

\begin{corollary}[Deformation jumping]\label{cor:intro-jumping}
There exist deformation-equivalent smooth projective surfaces $X$ and $X'$ such
that $\Om_X^1$ is big and hence
\[
 h^0\!\left(X,\Sym^m\Om_X^1\right)
\]
has cubic growth in $m$, whereas
\[
 H^0\!\left(X',\Sym^m\Om_{X'}^1\right)=0
 \qquad(m\ge1).
\]
In particular, the symmetric plurigenera exhibit an asymptotic jump and are
not deformation invariant.
\end{corollary}

\medskip
\noindent
\textbf{Organization.}
In \cref{sec:prelim} we recall weighted projective spaces in arbitrary dimension, weighted hypersurface stacks, reflexive symmetric powers, and the cohomology facts used later.  In \cref{sec:vanishing} we prove the characteristic-zero vanishing theorem \cref{thm:intro-vanishing} and the tame positive-characteristic theorem \cref{thm:intro-poschar}; we also give a weighted Calabi--Yau threefold example.  We then specialize to $k=\CC$ and $N=3$ for the surface examples.  From \cref{sec:bigness} onward all varieties are considered over this fixed field $k$ and all weighted hypersurfaces are surfaces in weighted projective three spaces.  In \cref{sec:Kummer} we study Kummer pullbacks and prove \cref{thm:intro-A5}.  The weighted Segre construction is carried out in \cref{sec:Segre}.  We conclude with the deformation-jumping phenomenon.

\section{Weighted hypersurfaces and reflexive differentials}
\label{sec:prelim}

Throughout the paper, $k$ denotes the base field.  In the vanishing part, $k$ is algebraically closed and its characteristic is specified when needed.  The weighted Calabi--Yau example is stated over the same field.  Starting with the surface examples following it, and throughout the bigness and deformation-theoretic part, we set $k=\CC$; thereafter the symbol $k$ always refers to this fixed complex base field.

\subsection{Weighted projective space}

Let $a_0,\dots,a_N$ be positive integers with $N\ge3$.  The weighted projective space
\[
 P=\PP_k(a_0,\ldots,a_N)
\]
is the coarse space associated with the quotient stack
\[
 \cP=
 \left[(\mathbb A_k^{N+1}\setminus\{0\})/\mathbb G_{m,k}\right],
\]
where
\[
 \lambda\cdot(x_0,\dots,x_N)
 =
 (\lambda^{a_0}x_0,\dots,\lambda^{a_N}x_N).
\]
We assume that $P$ is well formed, i.e.
\[
 \gcd(a_0,\dots,\widehat{a_i},\dots,a_N)=1
 \quad\text{for every }i.
\]
The stack $\cP$ is smooth and the coarse space $P$ has quotient singularities; see \cite{Dolgachev,Fletcher} and, for quotient stacks and coarse moduli spaces, \cite{AOV,Vistoli}.  For $q\in\mathbb Z$, we write $\OO_{\cP}(q)$ for the stack line bundle of weight $q$ and $\OO_P(q)$ for the corresponding rank-one reflexive sheaf on the coarse space.  The canonical class is
\[
 K_{\cP}=\OO_{\cP}\!\left(-\sum_{i=0}^N a_i\right).
\]
If $H=c_1(\OO_{\cP}(1))$, then
\[
 H^N=\frac1{a_0\cdots a_N}.
\]

\subsection{Quasi-smooth hypersurfaces and the hypersurface stack}

Let $F\in k[x_0,\ldots,x_N]$ be weighted homogeneous of degree $d$, and let $S=(F=0)\subset P$.  We say that $S$ is \emph{quasi-smooth} if the affine
cone $(F=0)\subset\mathbb A_k^{N+1}$ is smooth away from the origin.  Equivalently,
the associated hypersurface stack
\[
 \cS=
 \left[((F=0)\setminus\{0\})/\mathbb G_{m,k}\right]
 \subset\cP
\]
is smooth.  Its coarse space $S$ is normal and has at worst quotient singularities; see \cite{Dolgachev,Fletcher}.
We call $S$ \emph{well formed} if the coarse moduli map $p:\cS\to S$ is an isomorphism in codimension one; equivalently, $S$ contains no codimension-one stacky locus.

Weighted adjunction on $\cS$ gives
\begin{equation}\label{eq:adjunction}
 K_{\cS}
 =
 \OO_{\cS}\!\left(d-\sum_{i=0}^N a_i\right).
\end{equation}
When the coarse hypersurface is smooth and contained in $P_{\rm reg}$, this
is the usual adjunction formula on $S$.

\subsection{Reflexive symmetric differentials}

For a normal variety $S$ define
\[
 \Om_S^{[1]}:=\reflex{\Om_S^1},
 \qquad
 \Sym^{[m]}\Om_S^1:=
 \reflex{\Sym^m\Om_S^1}.
\]
If $j:S_{\rm reg}\hookrightarrow S$ denotes the inclusion, then
\[
 \Sym^{[m]}\Om_S^1
 \simeq
 j_*\Sym^m\Om_{S_{\rm reg}}^1.
\]
This is the usual codimension-two extension property for reflexive sheaves;
compare \cite{GKKP}.  For a well-formed quasi-smooth weighted hypersurface, the coarse moduli map
$p:\cS\to S$ is an isomorphism in codimension one.  In characteristic zero the stabilizers are linearly reductive.  In characteristic $p>0$ we will impose the tame hypothesis $p\nmid a_i$ for all $i$, which again makes every stabilizer finite and linearly reductive.  Let $U\subset S$ be the largest open set over which $p$ is an isomorphism.  By well-formedness, $S\setminus U$ has codimension at least two; moreover $U$ is smooth because $\cS$ is smooth.  On $U$ we therefore have
\[
 \left.p_*\Sym^m\Om_{\cS}^1\right|_U
 \simeq
 \Sym^m\Om_U^1.
\]
The local quotient description of tame stacks, together with the exactness of taking invariants for linearly reductive stabilizers, shows that $p_*\Sym^m\Om_{\cS}^1$ is reflexive; see \cite[Theorem~3.2]{AOV}.  Since a reflexive sheaf on the normal variety $S$ is determined by its restriction to the complement of a codimension-two subset, it follows that
\begin{equation}\label{eq:stack-reflexive}
 p_*\Sym^m\Om_{\cS}^1
 \simeq
 j_*\Sym^m\Om_U^1
 \simeq
 \Sym^{[m]}\Om_S^1,
\end{equation}
where $j:U\hookrightarrow S$ denotes the inclusion.  The same argument after twisting gives, for every integer $t$,
\[
 p_*\bigl(\Sym^m\Om_{\cS}^1\otimes\OO_{\cS}(t)\bigr)
 \simeq
 \left(\Sym^m\Om_S^1\otimes\OO_S(t)\right)^{\vee\vee}.
\]

\section{Vanishing of symmetric differentials}
\label{sec:vanishing}

We first record the line-bundle cohomology vanishings needed below.
\subsection{Line-bundle cohomology}

\begin{lemma}\label{lem:H1line}
Let $k$ be an algebraically closed field of arbitrary characteristic and let $\cS\subset\cP$ be a quasi-smooth weighted hypersurface of dimension at least two. Then
\[
 H^0(\cS,\OO_{\cS}(q))=0 \quad (q<0),
 \qquad
 H^1(\cS,\OO_{\cS}(q))=0 \quad (q\in\mathbb Z).
\]
In particular, the statement holds unchanged when $\operatorname{char}k=p>0$.
\end{lemma}

\begin{proof}
This is the stack version of the standard cohomology computation for weighted complete intersections; compare \cite[Theorem~3.2.4(iii) and Section~3.4.3]{Dolgachev} and \cite[Lemma~7.1]{Fletcher}. We include the short argument for completeness.

Let
$
 R=k[x_0,\ldots,x_N]/(F),
 \quad
 \deg x_i=a_i.
$
By the standard graded description,
\[
H^0(\mathcal S,\mathcal O_{\mathcal S}(q))\simeq R_q.
\]

Since \(R\) is positively graded, this vanishes for \(q<0\).

Since $R$ is a hypersurface ring, it is Cohen--Macaulay of dimension $N$.
Since $N\ge3$, for the irrelevant ideal $R_+$ we have
\[
 H^j_{R_+}(R)=0
 \qquad (j<N),
\]
in particular $H^2_{R_+}(R)=0$.
The standard graded local-cohomology description of cohomology on weighted
projective stacks gives (see also \cite{Dolgachev})
\[
 H^1(\cS,\OO_{\cS}(q))
 \simeq
 H^2_{R_+}(R)_q=0.
\]
\end{proof}

\begin{remark}\label{rem:H1-charfree}
The preceding lemma is characteristic-free.  In particular, the vanishing of $H^1(\cS,\OO_{\cS}(q))$ used below is available without any restriction on $\operatorname{char}k$.
\end{remark}

\begin{remark}\label{rem:qzero}
In particular,
\[
H^1(S,\OO_S)=0.
\]
When $N=3$, so that $S$ is a surface, this says that its irregularity is $q=0$.
\end{remark}

We now prove the vanishing theorem on the smooth hypersurface stack and then
pass to the coarse space to obtain \cref{thm:intro-vanishing}.

Set
\[
 W:=\sum_{i=0}^N a_i,
 \qquad
 a_{\min}:=\min_{0\le i\le N}a_i.
\]

\subsection{Characteristic zero}

In this subsection assume $\operatorname{char}k=0$.  All sequences below are sequences of vector bundles on the smooth
hypersurface stack $\cS$.  For a vector bundle $F$ on $\cS$ and $q\in\mathbb Z$, we write
\[
 F(q):=F\otimes\OO_{\cS}(q).
\]
Restrict the weighted Euler sequence to $\cS$:
\begin{equation}\label{eq:Euler}
 0\longrightarrow A
 \longrightarrow E
 \xrightarrow{\epsilon} \OO_{\cS}
 \longrightarrow0,
 \qquad
 E:=\bigoplus_{i=0}^N\OO_{\cS}(-a_i),
\end{equation}
where $A=\Om_{\cP}^1|_{\cS}$.  For every $r\ge1$, contraction with $\epsilon$
defines
\[
 c_{\epsilon}:\Sym^rE\longrightarrow\Sym^{r-1}E,
 \qquad
 c_{\epsilon}(e_1\cdots e_r)
 =\sum_{i=1}^r \epsilon(e_i)e_1\cdots\widehat{e_i}\cdots e_r.
\]
Locally choose a splitting \(E\simeq A\oplus \mathcal O_{\mathcal S}e\) such that \(\epsilon(e)=1\). Then

\[
\operatorname{Sym}^rE
\simeq
\bigoplus_{j=0}^r \operatorname{Sym}^{r-j}A\cdot e^j,
\]

and the contraction map satisfies

\[
c_\epsilon(ae^j)=j\,ae^{j-1},
\qquad
a\in \operatorname{Sym}^{r-j}A.
\]

Thus, after identifying \(e\) with a local polynomial variable, \(c_\epsilon\) is differentiation with respect to that variable. Since the ground field has characteristic zero, \(c_\epsilon\) is surjective and its kernel is \(\operatorname{Sym}^rA\).

Therefore
\begin{equation}\label{eq:symEuler}
 0\longrightarrow\Sym^rA
 \longrightarrow\Sym^rE
 \xrightarrow{c_{\epsilon}}\Sym^{r-1}E
 \longrightarrow0
\end{equation}
is exact.
The conormal sequence of the smooth hypersurface stack is
\begin{equation}\label{eq:conormal}
 0\longrightarrow\OO_{\cS}(-d)
 \longrightarrow A
 \longrightarrow\Om_{\cS}^1
 \longrightarrow0.
\end{equation}
Since $\OO_{\cS}(-d)$ is a line subbundle of $A$, the conormal sequence induces the exact sequence
\begin{equation}\label{eq:symconormal}
 0\longrightarrow
 (\Sym^{m-1}A)(-d)
 \longrightarrow
 \Sym^mA
 \longrightarrow
 \Sym^m\Om_{\cS}^1
 \longrightarrow0.
\end{equation}

\begin{theorem}\label{thm:stack-vanishing}
Assume $\operatorname{char}k=0$ and $d\ge2a_{\min}$.  For every $m\ge1$ and every integer $t<m a_{\min}$,
\[
 H^0\!\left(\cS,(\Sym^m\Om_{\cS}^1)(t)\right)=0.
\]
In particular, if $a_{\min}=1$, this gives the range $t<m$.
\end{theorem}

\begin{proof}
Every direct summand of $\Sym^mE$ is of the form
\[
 \OO_{\cS}(-\ell),
 \qquad
 \ell=\sum_{i=0}^Nk_ia_i,
 \quad
 \sum_{i=0}^Nk_i=m.
\]
Since $a_i\ge a_{\min}$, we have $\ell\ge m a_{\min}$.  Hence $t<m a_{\min}$ implies
$t-\ell<0$. By Lemma~\ref{lem:H1line},
\begin{equation}\label{eq:H0symA}
 H^0(\cS,(\Sym^mE)(t))=0,
 \qquad
 H^0(\cS,(\Sym^mA)(t))=0.
\end{equation}

Twisting \eqref{eq:symconormal} by $\OO_{\cS}(t)$, it remains to show
\[
 H^1\!\left(\cS,(\Sym^{m-1}A)(t-d)\right)=0.
\]
For $m=1$ this is Lemma~\ref{lem:H1line}.  Suppose $m\ge2$ and put
$r=m-1$.  Twist \eqref{eq:symEuler} by $\OO_{\cS}(t-d)$:
\[
 0\to(\Sym^rA)(t-d)
 \to(\Sym^rE)(t-d)
 \to(\Sym^{r-1}E)(t-d)
 \to0.
\]
Since $(\Sym^rE)(t-d)$ is a direct sum of line bundles, Lemma~\ref{lem:H1line} gives
\[
 H^1(\cS,(\Sym^rE)(t-d))=0.
\]
Every summand of $(\Sym^{r-1}E)(t-d)$ has degree
\[
 t-d-\ell',
 \qquad
 \ell'\ge r-1=m-2.
\]
Since $t<m a_{\min}$ and $t$ is an integer, we have
\[
t\le m a_{\min}-1.
\]
Moreover $\ell'\ge(m-2)a_{\min}$.  Hence
\[
t-d-\ell'
\le 2a_{\min}-1-d<0,
\]
because $d\ge2a_{\min}$.
Hence
\[
H^0(\cS,(\Sym^{r-1}E)(t-d))=0.
\]
The long exact cohomology sequence now gives
\[
 H^1(\cS,(\Sym^rA)(t-d))=0.
\]
Together with \eqref{eq:H0symA} and \eqref{eq:symconormal}, this proves the
claim.
\end{proof}

\begin{remark}
Note that we only use the condition $d\ge2a_{\min}$ in the last cohomological estimate.  No general-type assumption on the hypersurface is needed for the vanishing theorem.  
\end{remark}

\begin{corollary}\label{cor:coarse-vanishing}
Under the assumptions of \cref{thm:stack-vanishing}, suppose in addition that $S$ is well formed.
\[
 H^0\!\left(
 S,
 \left(\Sym^m\Om_S^1\otimes\OO_S(t)\right)^{\vee\vee}
 \right)=0
 \qquad (m\ge1,\ t<m a_{\min}).
\]
In particular $H^0(S,\Sym^{[m]}\Om_S^1)=0$ for all $m\ge1$.
\end{corollary}

\begin{proof}
Apply \eqref{eq:stack-reflexive} and take global sections.
\end{proof}

\begin{corollary}\label{cor:smooth-vanishing}
If in addition $S$ is smooth and contained in $P_{\rm reg}$, then
\[
 H^0\!\left(S,\Sym^m\Om_S^1\otimes\OO_S(t)\right)=0
 \qquad (m\ge1,\ t<m a_{\min}).
\]
\end{corollary}

\begin{proof}
Since $S\subset P_{\rm reg}$, the hypersurface stack $\cS$ is canonically identified with $S$, and $\OO_{\cS}(t)=\OO_S(t)$. The assertion follows directly from \cref{thm:stack-vanishing}.
\end{proof}

\begin{proof}[Proof of \cref{thm:intro-vanishing}]
Part~(ii) is \cref{cor:coarse-vanishing}, obtained from
\cref{thm:stack-vanishing} via the stack-to-coarse comparison
\eqref{eq:stack-reflexive}.  Part~(i) is \cref{cor:smooth-vanishing}.
\end{proof}

\subsection{Positive characteristic}

We now assume that $k$ is algebraically closed of characteristic $p>0$ and that
\[
p\nmid a_i\qquad\text{for all }i.
\]
Then the weighted projective stack and every quasi-smooth hypersurface stack considered below are tame Deligne--Mumford stacks.  In particular, the weighted Euler sequence remains a short exact sequence of vector bundles.

For $r\ge1$, let
\[
c_r:\Sym^rE\longrightarrow\Sym^{r-1}E
\]
be the Euler contraction defined as above, and put
\[
K_r:=\ker(c_r),\qquad I_r:=\operatorname{im}(c_r).
\]
The next lemma is a weighted analogue of the kernel filtration constructed by Gounelas--Liedtke in the proof of \cite[Proposition~8.1]{GounelasLiedtke}. 

\begin{lemma}\label{lem:weighted-liedtke}
For every $r\ge1$, the bundle $K_r$ admits a filtration
\[
\Sym^rA=F_0K_r\subset F_1K_r\subset\cdots\subset F_{\lfloor r/p\rfloor}K_r=K_r
\]
such that
\[
F_qK_r/F_{q-1}K_r\simeq\Sym^{r-qp}A
\qquad (1\le q\le\lfloor r/p\rfloor).
\]
\end{lemma}

\begin{proof}
This is a weighted adaptation of the kernel filtration in \cite[Proposition~8.1]{GounelasLiedtke}.  Locally split the weighted Euler sequence as
\[
E\simeq A\oplus\OO_{\cS}e,\qquad \epsilon(e)=1.
\]
Then
\[
\Sym^rE\simeq\bigoplus_{j=0}^r\Sym^{r-j}A\cdot e^j,
\qquad
c_r(ae^j)=j\,ae^{j-1}.
\]
Hence the kernel is locally the direct sum of the terms with $p\mid j$.  For $0\le q\le\lfloor r/p\rfloor$, define locally
\[
 F_qK_r:=\bigoplus_{h=0}^q\Sym^{r-hp}A\cdot e^{hp}.
\]
If $e'=e+u$ is another local splitting, with $u$ a local section of $A$, then in characteristic $p$
\[
(e')^p=e^p+u^p,
\qquad
(e')^{hp}=(e^p+u^p)^h.
\]
Thus $(e')^{hp}$ is a linear combination of $e^{jp}$ with $0\le j\le h$, whose remaining factors lie in symmetric powers of $A$.  It follows that each partial sum $F_qK_r$ is independent of the chosen splitting and hence glues globally.  Modulo $F_{q-1}K_r$, the term $(e')^{qp}$ differs from $e^{qp}$ only by lower filtration terms, so the coefficient of the top term is one.  Therefore
\[
 F_qK_r/F_{q-1}K_r\simeq\Sym^{r-qp}A,
\]
as claimed.
\end{proof}

\begin{theorem}\label{thm:poschar-stack}
Assume $d\ge2a_{\min}$.  For $m\ge1$, let $\rho_m$ be the remainder of $m-1$ modulo $p$.  Then
\[
H^0\!\left(\cS,(\Sym^m\Om_{\cS}^1)(t)\right)=0
\]
for every integer
\[
t<\min\{m a_{\min},\ d+\rho_m a_{\min}\}.
\]
In particular,
\[
H^0(\cS,\Sym^m\Om_{\cS}^1)=0
\qquad(m\ge1).
\]
If $m\le p$, the range is $t<m a_{\min}$.
\end{theorem}

\begin{proof}
The argument agrees with the characteristic-zero proof except for the Euler-contraction step.  The inequality $t<m a_{\min}$ still gives
\[
H^0(\cS,(\Sym^mA)(t))=0.
\]
It remains to prove
\[
H^1\!\left(\cS,(\Sym^{m-1}A)(t-d)\right)=0.
\]
For $m=1$ this is Lemma~\ref{lem:H1line}.  Let $r=m-1$ and $s=t-d$.  Since $I_r\subset\Sym^{r-1}E$, the assumptions $d\ge2a_{\min}$ and $t<m a_{\min}$ imply
\[
H^0(\cS,I_r(s))=0.
\]
Moreover Lemma~\ref{lem:H1line} gives
\[
H^1(\cS,(\Sym^rE)(s))=0.
\]
From
\[
0\to K_r(s)\to(\Sym^rE)(s)\to I_r(s)\to0
\]
we obtain $H^1(\cS,K_r(s))=0$.

Let $Q_r:=K_r/\Sym^rA$.  If $r<p$, then $Q_r=0$.  Otherwise, by Lemma~\ref{lem:weighted-liedtke}, $Q_r$ has successive quotients $\Sym^{r-qp}A$, and the smallest exponent occurring is $\rho_m$.  Since
\[
s=t-d<\rho_m a_{\min},
\]
and $A\hookrightarrow E$, all these successive quotients have no global sections after twisting by $\OO_{\cS}(s)$.  Thus in all cases $H^0(\cS,Q_r(s))=0$.  The sequence
\[
0\to(\Sym^rA)(s)\to K_r(s)\to Q_r(s)\to0
\]
therefore gives $H^1(\cS,(\Sym^rA)(s))=0$.  The conormal sequence now finishes the proof exactly as in characteristic zero.
\end{proof}

\begin{corollary}\label{cor:poschar-coarse}
Under the assumptions of \cref{thm:poschar-stack}, if $S$ is well formed, then
\[
H^0\!\left(
S,
\left(\Sym^m\Om_S^1\otimes\OO_S(t)\right)^{\vee\vee}
\right)=0
\]
for every
\[
t<\min\{m a_{\min},\ d+\rho_m a_{\min}\}.
\]
In particular $H^0(S,\Sym^{[m]}\Om_S^1)=0$ for every $m\ge1$.
\end{corollary}

\begin{proof}
Apply the tame stack-to-coarse comparison \eqref{eq:stack-reflexive} and take global sections.
\end{proof}

\begin{proof}[Proof of \cref{thm:intro-poschar}]
This is \cref{cor:poschar-coarse}.
\end{proof}

\begin{remark}[{\bf Spreading out to characteristic zero}]\label{rem:spread-out}
The positive characteristic theorem also gives an alternative proof of the
characteristic-zero vanishing theorem by spreading out.  We explain this on
the smooth hypersurface stack, thereby avoiding any compatibility issue for
reflexive hulls under reduction modulo $p$.

Let
\[
\cS_{\CC}\subset \PP_{\CC}(a_0,\ldots,a_N)
\]
be a quasi-smooth weighted hypersurface as in \cref{thm:stack-vanishing}, and
fix integers $m\ge1$ and $t<m a_{\min}$.  Choose a finitely generated integral
$\mathbb Z$-subalgebra $R\subset\CC$ over which the defining equation and the
weighted hypersurface stack are defined.  After replacing $R$ by a nonempty
open localization, we may assume that
\[
\cS_R\longrightarrow \Spec R
\]
is a smooth proper family.  We may then choose a closed point $s\in\Spec R$
whose residue field has characteristic
\[
p>\max\{m,a_0,\ldots,a_N\}.
\]
In particular, $p\nmid a_i$ for every $i$, and $m<p$.  Hence
\cref{thm:poschar-stack} applied to the fiber $\cS_s$ gives the full twisted
range
\[
H^0\!\left(\cS_s,
(\Sym^m\Om^1_{\cS_s})(t)\right)=0.
\]
Since
\[
\Sym^m\Om^1_{\cS_R/R}\otimes\OO_{\cS_R}(t)
\]
is a vector bundle on the smooth proper family, upper semicontinuity of
cohomology implies that the same vanishing holds on the generic fiber, and
hence, after base change, on $\cS_{\CC}$.  As $m$ and $t<m a_{\min}$ were
arbitrary, this recovers \cref{thm:stack-vanishing}.  In the well-formed case,
the reflexive statement on the coarse hypersurface then follows from the
characteristic-zero stack-to-coarse comparison \eqref{eq:stack-reflexive}.

Thus one may deduce the characteristic-zero theorem from the tame
positive-characteristic theorem.  We nevertheless keep the direct
characteristic-zero proof above because the Euler-contraction sequence is
shorter in characteristic zero and makes the origin of the optimal range
$t<m a_{\min}$ transparent.
\end{remark}

\begin{remark}\label{rem:BR-Liedtke}
Br\"uckmann--Rackwitz \cite{BR90} use Young symmetrizers and hence, for tensors of degree $m$, require characteristic zero or $p>m$.  Moreover, their theorem is formulated for complete intersections in ordinary projective space, so it does not directly give the weighted statement above.  Br\"uckmann's Frobenius-pullback results \cite{Brueckmann2002} extend this circle of ideas but retain the characteristic restriction.

There is also a separate obstruction to copying the characteristic-zero proof when $p\le r$: the kernel of the Euler contraction $c_r$ is then strictly larger than $\Sym^rA$.  Thus the short exact sequence \eqref{eq:symEuler} must be replaced by Lemma~\ref{lem:weighted-liedtke}, the weighted analogue of the $p$-power filtration in Gounelas--Liedtke \cite[Proposition~8.1]{GounelasLiedtke}.  After this replacement, the remaining line-bundle cohomology argument is characteristic-free by Lemma~\ref{lem:H1line} and Remark~\ref{rem:H1-charfree}.
\end{remark}

\subsection{A weighted Calabi--Yau threefold}

\begin{example}\label{ex:CY6}
Let $k$ be an algebraically closed field and let
\[
X_6\subset\PP_k(1,1,1,1,2)
\]
be a quasi-smooth hypersurface of weighted degree $6$.  Since
\[
6=1+1+1+1+2,
\]
weighted adjunction gives $K_{X_6}\simeq\OO_{X_6}$, so $X_6$ is a weighted Calabi--Yau threefold.  Here $a_{\min}=1$.  In characteristic zero, \cref{thm:stack-vanishing,cor:coarse-vanishing} give the twisted range $t<m$.  If $\operatorname{char}k=p>2$, then the weights are prime to $p$, hence \cref{thm:intro-poschar} applies to every quasi-smooth degree-$6$ hypersurface of this type; in particular
\[
H^0(X_6,\Sym^{[m]}\Om_{X_6}^1)=0
\qquad(m\ge1).
\]
For $m\le p$ one again has the full twisted range $t<m$.
\end{example}

\begin{remark}
The condition $p>2$ in \cref{ex:CY6} is the tameness condition coming from the weights.  It does not guarantee that an arbitrarily chosen equation is quasi-smooth.  For example, the Fermat equation
\[
z^3+x_0^6+x_1^6+x_2^6+x_3^6=0
\]
is not quasi-smooth in characteristic $3$.  Thus this particular model should be used only when $p\nmid6$, whereas the theorem itself applies to any quasi-smooth degree-$6$ hypersurface for every $p>2$.
\end{remark}

\medskip
From this point until the end of the present section, we set $k=\CC$ and specialize to $N=3$; thus all hypersurfaces below are surfaces over the fixed field $k$ in weighted projective three spaces.

\subsection{A smooth Horikawa hypersurface}

\begin{example}\label{ex:Horikawa}
Consider
\[
 X_8=
 \{x_0^8+x_1^8+x_2^8+z^2=0\}
 \subset\PP(1,1,1,4).
\]
The ambient weighted projective space has a unique singular point
$[0:0:0:1]$, which does not lie on $X_8$.  The affine cone of $X_8$ is
smooth away from the origin, hence $X_8$ is a smooth surface contained in
$P_{\rm reg}$.  Weighted adjunction gives
$K_{X_8}=\OO_{X_8}(1)$, and
\[
 K_{X_8}^2=2,
 \qquad
 p_g(X_8)=3,
 \qquad
 q(X_8)=0.
\]
In particular
\[
 K_{X_8}^2=2p_g(X_8)-4,
\]
so $X_8$ is a Horikawa surface on the Noether line.  Since
$8>1+1+1+4$ and $a_{\min}=1$, \cref{cor:smooth-vanishing} gives
\[
H^0\!\left(
X_8,\Sym^m\Om_{X_8}^1\otimes\OO_{X_8}(t)
\right)=0
\qquad (m\ge1,\ t<m).
\]

\end{example}

\begin{remark}
Brotbek--De Oliveira--Rousseau \cite{BOR2026} study symmetric
differentials on double covers of rational surfaces and, in particular, on
Horikawa surfaces.  The surface \(X_8\) above is a Horikawa surface with
\(p_g=3\), realized as a smooth double cover of \(\PP^2\) branched along a
smooth octic. In the notation of \cite{BOR2026}, it belongs to the
\(\PP^2\)-double-cover stratum \(\mathcal H_3^{(\infty)}\).
Their vanishing theorem for smooth double covers of \(\PP^2\) therefore
implies

\[
H^0(X_8,\Sym^m\Om_{X_8}^1)=0
\qquad (m\ge1).
\]

For this example, the vanishing obtained here is stronger with respect to the natural weighted polarization
\(\OO_{X_8}(1)=K_{X_8}\).
\end{remark}

\subsection{Naive K\"ahler differentials: a counterexample}
\label{subsec:torsion-example}

We now show that the reflexive hull in
\cref{cor:coarse-vanishing} cannot in general be omitted.

\begin{example}\label{ex:torsion}
Let
\[
 S=
 \{z_3^3z_0+z_2^8+z_0^{16}+z_1^{16}=0\}
 \subset\PP(1,1,2,5).
\]
Then $S$ is quasi-smooth.  Indeed, the partial derivatives of the affine cone
equation are
\[
 z_3^3+16z_0^{15},\qquad
 16z_1^{15},\qquad
 8z_2^7,\qquad
 3z_3^2z_0,
\]
and they vanish simultaneously only at the origin.

The point $P_2=[0:0:1:0]$ does not lie on $S$, while
$P_3=[0:0:0:1]$ does.  On the index-one chart over $P_3$ the group
$\mu_5$ acts by
\[
 \zeta\cdot(x_0,x_1,x_2)
 =
 (\zeta x_0,\zeta x_1,\zeta^2x_2),
\]
and the lifted hypersurface is
\[
 x_0+x_2^8+x_0^{16}+x_1^{16}=0.
\]
Since the derivative with respect to $x_0$ is nonzero at the origin, the
completed local germ of the coarse surface is
\[
 \widehat{\OO}_{S,P_3}
 \simeq
 k[[u,v]]^{\mu_5},
 \qquad
 \zeta\cdot(u,v)=(\zeta u,\zeta^2v).
\]
Thus $P_3$ is a cyclic quotient singularity of type $\frac15(1,2)$ and is the
unique singular point of $S$.

We now exhibit torsion in the K\"ahler differential module of this quotient.
Put
\[
 R=k[u,v]^{\mu_5}.
\]
The invariant ring is generated by
\[
 X=u^5,\qquad
 Y=u^3v,\qquad
 Z=uv^2,\qquad
 T=v^5
\]
with relations
\[
 Y^2-XZ=0,\qquad
 Z^3-YT=0,\qquad
 YZ^2-XT=0.
\]
In $\Om_R^1$ consider
\begin{equation}\label{eq:torsion-tau}
 \tau:=T\,dX-2Z^2\,dY+YZ\,dZ.
\end{equation}
After pulling back to $k[u,v]$ we obtain
\[
\begin{aligned}
 \tau
 &=
 v^5(5u^4\,du)
 -2u^2v^4(3u^2v\,du+u^3\,dv)  \\
 &\qquad
 +u^4v^3(v^2\,du+2uv\,dv)
 =0.
\end{aligned}
\]
Since $k(u,v)$ is a finite separable extension of $\operatorname{Frac}(R)$,
the kernel of
\[
 \Om_R^1\longrightarrow\Om_{k[u,v]}^1
\]
is torsion.  Hence $\tau$ is torsion.

It remains to note that $\tau\neq0$ in $\Om_R^1$.  Give $R$ the
$\mathbb Z^2$-grading
\[
 \deg X=(5,0),\quad
 \deg Y=(3,1),\quad
 \deg Z=(1,2),\quad
 \deg T=(0,5).
\]
Then $\tau$ has bidegree $(5,5)$.  The three defining relations have
bidegrees $(6,2)$, $(3,6)$ and $(5,5)$, respectively.  In bidegree $(5,5)$
the conormal relations are therefore spanned only by
\[
 d(YZ^2-XT)
 =
 Z^2\,dY+2YZ\,dZ-T\,dX-X\,dT,
\]
and \eqref{eq:torsion-tau} is not a scalar multiple of this element.  Thus
$\tau\neq0$.

By faithful flatness of completion, the stalk
$\Tor(\Om_S^1)_{P_3}$ is nonzero.  Since the torsion sheaf is supported at the
single point $P_3$, it follows that
\[
 H^0(S,\Om_S^1)\neq0.
\]
On the other hand, $16>1+1+2+5$ and
\cref{cor:coarse-vanishing} gives
\[
 H^0(S,\Om_S^{[1]})=0.
\]
Thus the nonzero global K\"ahler differentials are purely torsion and
disappear after reflexive hull.
\end{example}

\begin{remark}
For the surface in \cref{ex:torsion},
\[
 K_S=\OO_S(7),\qquad
 p_g(S)=24,\qquad
 q(S)=0,\qquad
 \chi(\OO_S)=25.
\]
The minimal resolution of the $\frac15(1,2)$ singularity has exceptional
chain
\[
 E_1^2=-3,\qquad E_2^2=-2,\qquad E_1E_2=1,
\]
and
\[
 K_{\widetilde S}
 =
 \pi^*K_S-\frac25E_1-\frac15E_2.
\]
Consequently
\[
 K_{\widetilde S}^2=78,
 \qquad
 c_2(\widetilde S)=222.
\]
\end{remark}

\section{The quotient-singularity bigness criterion}
\label{sec:bigness}

From this section onward we keep $k=\CC$ and $N=3$; thus all weighted hypersurfaces considered below are surfaces over the fixed field $k$ in weighted projective three spaces.  This specialization is essential for the general-type surface theory, the quotient-singularity criterion, the Chern-number computations, and the Kummer, Segre, and deformation applications below.

Let $Y$ be a normal projective surface with quotient singularities and let
$\sigma:\widetilde Y\to Y$ be the minimal resolution.  For
$y\in\Sing(Y)$, choose a sufficiently small neighborhood germ $U_y$ and let
$\widetilde U_y\to U_y$ be its minimal resolution.  Following
Asega--De Oliveira--Weiss, define
\[
 h_\Omega^1(y)
 :=
 \liminf_{m\to\infty}
 \frac{
 h^0\!\left(
 U_y,
 R^1\sigma_*\Sym^m\Om_{\widetilde U_y}^1
 \right)
 }{m^3}.
\]

\begin{theorem}[QS-bigness criterion {\cite[Theorem 1.1]{ADW1}}]
\label{thm:QS}
Assume that $\widetilde Y$ is of general type.  If
\[
 \sum_{y\in\Sing(Y)}h_\Omega^1(y)
 >
 \frac{c_2(\widetilde Y)-c_1^2(\widetilde Y)}6,
\]
then $\Om_{\widetilde Y}^1$ is big.
\end{theorem}

For $A_n$ singularities, the following formula
is proved in \cite{ADW1,ADW2}:
\begin{equation}\label{eq:A1A5}
 h_\Omega^1(A_1)=\frac4{27},
 \qquad
 h_\Omega^1(A_5)=\frac{106819}{132300}.
\end{equation}

\begin{corollary}\label{cor:nodes}
Suppose $Y$ has exactly $\ell$ singular points, all nodes, and no other
singularities.  If $\widetilde Y$ is of general type and
\[
 \ell>\frac98\bigl(c_2(\widetilde Y)-c_1^2(\widetilde Y)\bigr),
\]
then $\Om_{\widetilde Y}^1$ is big.
\end{corollary}

\begin{proof}
Use $h_\Omega^1(A_1)=4/27$ in \cref{thm:QS}.
\end{proof}

\section{Kummer pullbacks}
\label{sec:Kummer}

\subsection{The Kummer map}

Let
$
 P=\PP(a_0,a_1,a_2,a_3)
$
with pairwise coprime weights and not all $a_i=1$. Put
\[
 L=\operatorname{lcm}(a_0,a_1,a_2,a_3)
 =a_0a_1a_2a_3.
\]
Consider
\begin{equation}\label{eq:Kummer}
 \Phi:P\longrightarrow\PP^3,
 \qquad
 [x_0:x_1:x_2:x_3]
 \longmapsto
 [x_0^{L/a_0}:x_1^{L/a_1}:x_2^{L/a_2}:x_3^{L/a_3}].
\end{equation}
Every coordinate on the right has weighted degree $L$.  If $H$ denotes the
weighted hyperplane class, then
\[
 H^3=\frac1L,
 \qquad
 \Phi^*\OO_{\PP^3}(1)=\OO_P(L),
\]
and hence
\begin{equation}\label{eq:degreePhi}
 \deg\Phi=(LH)^3=L^2.
\end{equation}

After a general projective transformation of a fixed surface
$X_0\subset\PP^3$, we may and do assume that all singular points of $X_0$
avoid the coordinate hyperplanes and that the smooth locus of $X_0$ meets
the coordinate strata transversely.

\begin{lemma}\label{lem:Kummer-singularities}
Under the preceding assumptions, $Y:=\Phi^{-1}(X_0)$ has no singularities over the smooth locus of $X_0$.  Moreover, every singular point of $X_0$ has exactly $L^2$ preimages, and $\Phi$ is \'etale at each of them.  In particular, those preimages have the same analytic singularity type as the corresponding point of $X_0$.
\end{lemma}

\begin{proof}
Away from the coordinate hyperplanes the map $\Phi$ is \'etale.  Since every singular point of $X_0$ lies in this open torus, \eqref{eq:degreePhi} shows that it has exactly $L^2$ preimages, and \'etaleness preserves the analytic singularity type.

It remains to consider a point $y\in Y$ mapping to a smooth point $x\in X_0$ on a coordinate stratum.  The transversality assumption implies in particular that $x$ does not lie on a zero-dimensional coordinate stratum.  On the smooth locus of $P$, and after passing to \'etale local coordinates at $y$ and $x$, the Kummer map has the standard monomial form
\[
 (u_1,\ldots,u_s,v_{s+1},\ldots,v_3)
 \longmapsto
 (u_1^{e_1},\ldots,u_s^{e_s},v_{s+1},\ldots,v_3),
\]
where the coordinate stratum through $x$ is given by the vanishing of the first $s$ target coordinates.  Write $f=0$ for a local equation of $X_0$.  Transversality of $X_0$ to that stratum means that the restriction of $df_x$ to the tangent space of the stratum is nonzero.  Hence
\[
 \frac{\partial f}{\partial v_j}(x)\ne0
\]
for some $j>s$.  For the pullback equation $f\circ\Phi=0$, the corresponding derivative at $y$ is the same nonzero number.  Thus $Y$ is smooth at $y$.  This proves the lemma.
\end{proof}

For later use set
\[
 W:=\sum_{i=0}^3a_i,
 \qquad
 E_2:=\sum_{0\le i<j\le3}a_ia_j.
\]

\begin{proposition}\label{prop:Kummer-criterion}
Let $X_0\subset\PP^3$ be a degree-$t$ surface, $t\ge5$, with only
Du Val singularities $x_1,\dots,x_\ell$.  Assume that, after a general
projective transformation, the transversality conditions above hold.
Let $Y=\Phi^{-1}(X_0)$ and $\widetilde Y\to Y$ be the minimal resolution.
If
\begin{equation}\label{eq:Kummer-hyp}
 \sum_{i=1}^{\ell}h_\Omega^1(x_i)
 >
 \frac{t}{6}\left(\frac52t-4\right),
\end{equation}
then $\Om_{\widetilde Y}^1$ is big.
\end{proposition}

\begin{proof}
The pullback $Y$ has weighted degree $tL$ and, by \cref{lem:Kummer-singularities}, exactly $L^2$ preimages of each singularity and no other singularities.  Since the singularities are Du Val, the minimal resolution is
crepant and, after simultaneous resolution of a smoothing, has the same
Chern numbers as a smooth weighted hypersurface of degree $tL$.  Hence
\begin{align}
 c_1^2(\widetilde Y)
 &=t(tL-W)^2,\\
 c_2(\widetilde Y)
 &=t\bigl(t^2L^2-tLW+E_2\bigr),
\end{align}
and therefore
\begin{equation}\label{eq:chern-difference}
 c_2(\widetilde Y)-c_1^2(\widetilde Y)
 =
 t\bigl(tLW+E_2-W^2\bigr).
\end{equation}

Set $u_i=a_i/L$.  For a genuinely weighted space we may order the weights so
that
\[
 0<u_0,u_1,u_2\le\frac12,
 \qquad
 0<u_3\le1.
\]
Then
\[
 \frac{tLW+E_2-W^2}{L^2}
 =
 f(u_0,u_1,u_2,u_3),
\]
where
\[
 f(u)
 =
 t\sum_i u_i+\sum_{i<j}u_iu_j-\left(\sum_i u_i\right)^2.
\]
For $t\ge5$,
\[
 \frac{\partial f}{\partial u_i}
 =
 t-\sum_j u_j-u_i>0
\]
on the above box.  Consequently
\[
 f(u)
 \le
 f\!\left(\frac12,\frac12,\frac12,1\right)
 =
 \frac52t-4.
\]
By \eqref{eq:chern-difference},
\[
 \frac{c_2(\widetilde Y)-c_1^2(\widetilde Y)}{6}
 \le
 \frac{tL^2}{6}\left(\frac52t-4\right).
\]
On the other hand the local contribution on $Y$ is
\[
 L^2\sum_{i=1}^{\ell}h_\Omega^1(x_i).
\]
The claim follows from \eqref{eq:Kummer-hyp} and \cref{thm:QS}.
\end{proof}

\subsection{A degree \texorpdfstring{$12$}{12} hypersurface with \texorpdfstring{$60$}{60} singularities of type \texorpdfstring{$A_5$}{A5}}

Asega--De Oliveira--Weiss consider the degree-$d$ cyclic cover of $\PP^2$
branched along $d$ lines in general position \cite[Proof of
Theorem~1.5]{ADW1}.  For $d=6$, choose six lines
$L_1,\ldots,L_6\subset\PP^2$ in general position and write the cover as
\begin{equation}\label{eq:cyclic-sextic}
 X_6=
 \left\{
 w^6=L_1(x_0,x_1,x_2)\cdots L_6(x_0,x_1,x_2)
 \right\}
 \subset\PP^3.
\end{equation}
Since no three of the lines are concurrent, the singular points of $X_6$
lie precisely over the $\binom62=15$ pairwise intersections of the branch
lines.  Near such a point the equation is analytically equivalent to
\[
 uv-w^6=0,
\]
so $X_6$ has exactly $15$ singularities of type $A_5$ and no others.

After a general projective transformation, assume that these singular points
avoid the coordinate hyperplanes, that $[0:0:0:1]\notin X_6$, and that the
smooth locus meets all coordinate strata transversely.  Consider
\[
 \Phi_2:\PP(1,1,1,2)\longrightarrow\PP^3,
 \qquad
 [x_0:x_1:x_2:w]
 \longmapsto
 [x_0^2:x_1^2:x_2^2:w].
\]
Its degree is $4$.  Put
\[
 Y_{12}:=\Phi_2^{-1}(X_6).
\]
Then $Y_{12}$ has weighted degree $12$, avoids the unique singular point of
$\PP(1,1,1,2)$, and has exactly $60$ singularities of type $A_5$ and no
other singularities.

\begin{theorem}\label{thm:A5example}
Let $\widetilde Y_{12}\to Y_{12}$ be the minimal resolution.  Then
$\Om_{\widetilde Y_{12}}^1$ is big.
\end{theorem}

\begin{proof}
By \eqref{eq:A1A5},
\[
 15\,h_\Omega^1(A_5)
 =15\cdot\frac{106819}{132300}
 =\frac{106819}{8820}
 >11
 =\frac{6}{6}\left(\frac52\cdot6-4\right).
\]
Thus $X_6$ satisfies the numerical hypothesis of
\cref{prop:Kummer-criterion}.  Applying that proposition to $\Phi_2$ gives
the claim.
\end{proof}

\subsection{A nodal Kummer example from the Barth decic}

For completeness we record another family.  Barth constructed a degree-ten
surface $B_{10}\subset\PP^3$ with $345$ nodes \cite{Barth}.  After a general
projective transformation, define
\[
 X_r=\Phi_r^{-1}(B_{10})
 \subset\PP(1,1,1,r),
 \qquad
 \Phi_r[x_0:x_1:x_2:w]
 =
 [x_0^r:x_1^r:x_2^r:w].
\]
Then $X_r$ has degree $10r$ and $345r^2$ nodes, and no other
singularities.  If $Y_r\to X_r$ is the minimal resolution, then
\[
 K_{Y_r}^2=10(9r-3)^2,
 \qquad
 c_2(Y_r)=900r^2-270r+30.
\]
Thus
\[
 c_2(Y_r)-K_{Y_r}^2=90r^2+270r-60.
\]
For every $r\ge2$,
\[
 345r^2>
 \frac98\bigl(90r^2+270r-60\bigr),
\]
so \cref{cor:nodes} gives:

\begin{corollary}\label{cor:Barth}
For every $r\ge2$, the minimal resolution of $X_r$ has big cotangent
bundle.
\end{corollary}

\section{A weighted Segre construction}
\label{sec:Segre}

The construction below is a weighted analogue of Segre's classical
construction of nodal surfaces \cite{Segre}.

Fix $r\ge2$ and put
\[
 P_r=\PP(1,1,1,r).
\]
Let $q\ge1$ and set
\[
 d=2qr.
\]
Choose $2qr$ general lines
\[
 L_1,\dots,L_{2qr}\subset\PP^2=\{w=0\}\subset P_r
\]
and choose a general polynomial
\[
 G\in H^0(P_r,\OO_{P_r}(qr))
\]
such that $G$ does not vanish at the singular point
$[0:0:0:1]\in P_r$.  For a general nonzero scalar $\lambda$, define
\begin{equation}\label{eq:Segre}
 X_d=
 \left\{
 G^2+\lambda\prod_{i=1}^{2qr}L_i=0
 \right\}
 \subset P_r.
\end{equation}
The hypersurface avoids the ambient singular point. For general choices, Bertini's theorem shows that it is smooth away from the base locus of the pencil generated by $G^2$ and $\prod_i L_i$. Along the base locus, a singular point must satisfy $G=0$ and lie on at least two of the lines. Since the lines are general, no three are concurrent. Near a point with $G=L_i=L_j=0$, the functions $G,L_i,L_j$ may be taken as local coordinates up to units, and the equation is analytically of the form
\[
 g^2+\lambda uv\cdot(\text{unit})=0,
\]
which is an ordinary double point. Thus the singular locus consists precisely of the points at which $G=0$ and two of the lines $L_i$ meet, and all these singularities are nodes.

For every pair $L_i,L_j$, the complete intersection $L_i=L_j=0$ is
isomorphic to $\PP(1,r)$, and
\[
 \deg\bigl(\OO_{\PP(1,r)}(qr)\bigr)=q.
\]
Thus each pair contributes $q$ nodes and the total number of nodes is
\begin{equation}\label{eq:numbernodes}
 \ell
 =
 q\binom{2qr}{2}
 =
 q^2r(2qr-1).
\end{equation}

Let $\widetilde X_d\to X_d$ be the minimal resolution.  Since the
singularities are nodes, it is crepant and is deformation equivalent to a
smooth hypersurface of degree $d$.  The corresponding Chern-number formulas are
\begin{align}
 K_{\widetilde X_d}^2
 &=
 \frac{d(d-r-3)^2}{r}
 =
 2q(2qr-r-3)^2,
 \label{eq:SegreK2}\\
 c_2(\widetilde X_d)
 &=
 \frac{d\bigl(d^2-d(r+3)+3r+3\bigr)}{r}\notag\\
 &=
 2q\bigl(
 4q^2r^2-2qr(r+3)+3r+3
 \bigr).
 \label{eq:Segrec2}
\end{align}

\begin{proposition}\label{prop:Segre}
If
\begin{equation}\label{eq:qbound}
 q>
 \frac{9r+29+\sqrt{9r^2+306r+409}}{8r},
\end{equation}
then $\Om_{\widetilde X_d}^1$ is big.
\end{proposition}

\begin{proof}
By \cref{cor:nodes}, it is enough to prove
\[
 \ell>\frac98
 \left(c_2(\widetilde X_d)-K_{\widetilde X_d}^2\right).
\]
Using \eqref{eq:numbernodes}, \eqref{eq:SegreK2} and
\eqref{eq:Segrec2}, this is equivalent to
\[
 8\ell-9(c_2-K^2)>0.
\]
A direct simplification gives
\begin{equation}\label{eq:Segrepoly}
 8\ell-9(c_2-K^2)
 =
 2q\left[
 8r^2q^2-(18r^2+58r)q+9r^2+27r+54
 \right].
\end{equation}
The larger root of the quadratic polynomial in brackets is
\[
 \frac{9r+29+\sqrt{9r^2+306r+409}}{8r}.
\]
Hence \eqref{eq:qbound} implies \eqref{eq:Segrepoly} is positive.

Finally,
\[
 K_{X_d}=\OO_{X_d}(2qr-r-3).
\]
For $q$ satisfying \eqref{eq:qbound}, in particular $2qr>r+3$, so
$K_{X_d}$ is ample.  Thus $\widetilde X_d$ is of general type, as required
by the QS-bigness criterion.
\end{proof}

\begin{corollary}
For every $r\ge2$ there exist nodal hypersurfaces in
$\PP(1,1,1,r)$ whose minimal resolutions have big cotangent bundle.
\end{corollary}

\section{Deformation and jumping of symmetric plurigenera}
\label{sec:jumping}

\begin{proposition}\label{prop:duval-deformation}
Let $r\ge2$ and let
\[
 X_0\subset\PP(1,1,1,r)
\]
be a hypersurface of weighted degree $d$ with $r\mid d$.  Assume that $X_0$ is contained in the smooth locus of the ambient weighted projective space and has only Du Val singularities.  Then its minimal resolution is deformation equivalent to a smooth weighted hypersurface of degree $d$.
\end{proposition}

\begin{proof}
Write $X_0=(F_0=0)$.  Since $r\mid d$, the sheaf $\OO_{\PP(1,1,1,r)}(d)$ is globally generated: the monomials $x_0^d,x_1^d,x_2^d$ and $w^{d/r}$ have no common zero.  Moreover the linear system contains smooth members contained in the smooth locus; for example
\[
 x_0^d+x_1^d+x_2^d+w^{d/r}=0
\]
is quasi-smooth and avoids the singular point $[0:0:0:1]$.  We may therefore choose a general section $F_1$ such that $X_1=(F_1=0)$ is smooth, contained in the smooth locus, and $F_1$ does not vanish at any singular point of $X_0$.

Consider the pencil
\[
 X_t=(F_0+tF_1=0).
\]
At a singular point $x$ of $X_0$, the total space of the pencil is smooth because the derivative in the parameter direction is $F_1(x)\ne0$.  Away from the singular points of $X_0$, smoothness of the total space along the central fiber follows from the smoothness of $X_0$ there.  Since the pencil contains the smooth member $X_1$, after restricting to a sufficiently small disc around $0$ we may assume that $X_t$ is smooth for every $t\ne0$.  Thus the pencil is a one-parameter smoothing of the Du Val surface $X_0$ inside the degree-$d$ hypersurface linear system.

After a finite base change, Brieskorn simultaneous resolution \cite{Brieskorn} gives a smooth family whose central fiber is the minimal resolution of $X_0$ and whose nearby fibers are the smooth weighted hypersurfaces $X_t$.  Hence the minimal resolution of $X_0$ is deformation equivalent to a smooth weighted hypersurface of degree $d$.
\end{proof}

\begin{proof}[Proof of \cref{cor:intro-jumping}]
For the degree-$12$ surface in $\PP(1,1,1,2)$ from \cref{thm:A5example}, apply \cref{prop:duval-deformation} with $r=2$.  The same proposition applies to every degree-$2qr$ nodal surface in $\PP(1,1,1,r)$ from \cref{sec:Segre}.
Their minimal resolutions have big cotangent bundle, so
\[
 h^0(\widetilde X_0,\Sym^m\Om_{\widetilde X_0}^1)
\]
has cubic growth in $m$.  For a smooth weighted hypersurface $X_t$ in the
same deformation class, \cref{cor:smooth-vanishing} gives
\[
 H^0(X_t,\Sym^m\Om_{X_t}^1)=0
 \qquad(m\ge1).
\]
This proves the claim.
\end{proof}

\begin{remark}
The surfaces constructed above provide further examples of surfaces satisfying
the Green--Griffiths--Lang conjecture.  Indeed, by the work of Bogomolov and
McQuillan, a smooth projective surface of general type with big cotangent
bundle satisfies the Green--Griffiths--Lang conjecture; namely, all entire
curves are contained in a proper algebraic subvariety
\cite{Bogomolov,McQuillan}; see also \cite{ADW1}.  Consequently, the minimal
resolution of the degree-$12$ hypersurface in \cref{thm:A5example} and the
minimal resolutions of the weighted Segre hypersurfaces in
\cref{prop:Segre} satisfy the Green--Griffiths--Lang conjecture.
\end{remark}

\section*{Acknowledgements}

The author thanks C. Liedtke, Xin Lu, Sheng-Li Tan and Hao Sun for helpful discussions. In particular, it was Liedtke who brought the work \cite{GounelasLiedtke} to the author, which led the author to consider the positive-characteristic case.
The author was partially supported by the National Natural Science
Foundation of China and by the Science and Technology Innovation Plan of
Shanghai (Grant No.\ 23JC1403200).

{\bf AI Declaration:} The central mathematical ideas and overall approach are due to the author. AI tools assisted with manuscript preparation and English revision. The author takes full responsibility for the paper's content and the accuracy of its references.


\end{document}